\documentclass[12pt]{article}
\usepackage{amstext,amssymb,amsmath,amsbsy}

\usepackage{hyperref}
\usepackage{amscd}
\usepackage{amsfonts}
\usepackage{indentfirst}
\usepackage{verbatim}
\usepackage{amsmath}
\usepackage{amsthm}
\usepackage{enumerate}
\usepackage{graphicx}
\usepackage[OT1]{fontenc}
\usepackage[latin1]{inputenc}
\usepackage[english]{babel}
\usepackage{amssymb}

\newcommand{\nm}{\noalign{\smallskip}}
\newcommand{\ds}{\displaystyle}
\newcommand{\D}{D}
\newcommand{\Dx}{D}
\newcommand{\beq}{\begin{equation}}
\newcommand{\eeq}[1]{\label{#1}\end{equation}}
\newcommand{\Beq}{\begin{equation} \left\{ \begin{array}{rcll}}
\newcommand{\Eeq}[1]{ \end{array} \right.\label{#1}\end{equation}}
\newcommand{\intome}{\int_{\underline \omega}^{\overline \omega}}
\newcommand{\Div}{\nabla \cdot}
\newcommand{\admitt}{\sigma + i \omega \varepsilon}
\newcommand{\admittj}{\sigma + i \omega_j \varepsilon}
\newcommand{\admittO}{\sigma + i \omega_0 \varepsilon}

\newtheorem{Lemma}{Lemma}[section]
\newtheorem{Proposition}{Proposition}[section]
\newtheorem{Corollary}{Corollary}[section]
\newtheorem{Theorem}{Theorem}[section]

\newtheorem{Definition}{Definition}[section]
\numberwithin{equation}{section} \numberwithin{figure}{section}

\title{Admittivity imaging from multi-frequency micro-electrical impedance tomography\thanks{\footnotesize This work was supported  by the
ERC Advanced Grant Project MULTIMOD--267184.}}

\author{Habib Ammari\thanks{\footnotesize Department of Mathematics and Applications,
Ecole Normale Sup\'erieure, 45 Rue d'Ulm, 75005 Paris, France
(habib.ammari@ens.fr, laure.giovangigli@ens.fr).} \and Laure Giovangigli\footnotemark[2] \and
Loc Hoang Nguyen\thanks{Mathematics Section, \'Ecole Polytechnique F\'ed\'erale de Lausanne,
Station 8,  CH-1015 Lausanne, Switzerland (loc.nguyen@epfl.ch).} \and Jin-Keun
Seo\thanks{\footnotesize Department of Computational Science and
Engineering, Yonsei University 50 Yonsei-Ro, Seodaemun-Gu, Seoul
120-749, Korea (seoj@yonsei.ac.kr).}}

\begin{document}

\maketitle


%

\begin{abstract}
The aim of this paper is to propose an optimal control optimization algorithm for reconstructing admittivity distributions (i.e., both conductivity and permittivity) from multi-frequency micro-electrical impedance tomography.  A convergent and stable optimization scheme is shown  to be obtainable from multi-frequency data.
The results of this paper have potential  applicability in cancer
imaging, cell culturing and differentiation, food sciences, and biotechnology.
\end{abstract}

\bigskip

\noindent {\footnotesize Mathematics Subject Classification
(MSC2000): 35R30, 35B30.}


\maketitle



\selectlanguage{english}

\section{Introduction}
This paper aims at proposing and analyzing an optimal control approach for imaging the admittivity distributions of biological tissues.

Biological tissues possess characteristic distributions of permittivity and  conductivity \cite{bio3}. Conductivity can be regarded as a measure of the ability to transport charge throughout material's volume under an applied electric field, while permittivity is a measure of the ability of the dipoles within a material to rotate (or of the material to store charge) under an applied external field.
At low frequencies, biological tissues behave like a conductor, but capacitive effects become important at higher frequencies due to the membranous structures \cite{schwan3, jinkeun}. The electric behavior of a biological tissue under the influence of
an electric field at frequency $\omega$ can be characterized by its frequency-dependent admittivity $\sigma + i \omega \varepsilon$, where $\sigma$ and $\varepsilon$ are respectively its conductivity and permittivity.  As recently shown in \cite{laure, jinkeun}, spectroscopic admittivity imaging can provide information about the microscopic structure of a medium, from which physiological or pathological conditions of tissue can be derived, because the admittivity of biological tissue varies with its composition, membrane characteristics, intra-and extra-cellular fluids, and other factors.

In this paper, we consider the imaging of admittivity distributions of biological tissues from multi-frequency micro-electrical impedance data.  Micro-electrical impedance tomography  \cite{microeit, pnas} can be used to reconstruct a high resolution admittivity distribution from internal measurements of  electrical potential at multiple frequencies.
 The technique uses  planar arrays of micro-electrodes  to nondestructively sense thin layers of biological samples \cite{microeit2, microeit, microeit3, microeit4, microeit27}. It has potential applications in cell electrofusion and electroporation, cell culturing, cell differentiation and drug screening; see \cite{food, microeit,  keese3, keese4, biotech, keese5, blood,culture,food2}. It is capable of high-resolution  imaging. Other methods of electrical tissue property imaging using internal data are investigated in \cite{AMMARI-08,  AMMARI-BONNETIER-CAPDEBOSCQ-08, siap2011, ejam,  pol, GEBAUER-SCHERZER-08, sirev, seobook, otmar2}.  Resolution and stability enhancements are achieved from internal measurements \cite{ip12,AmmariWenjiaLoc:arXiv12010619v2, AmmariLocLaurentetal:preprint2012}.

To solve the admittivity imaging problem from multi-frequency micro-electrical data, we propose an optimal control optimization algorithm  and rigorously prove its stability and convergence properties.
 Internal potential measurements at a single frequency are known to be insufficient for  reconstructing the admittivity distribution. An initial guess is constructed by solving a boundary value problem. Note that the method of characteristics \cite{KwonLeeYoon:ip2002} can not be used  to solve the transport equation satisfied by the logarithm of the admittivity, because the electrical potential is a complex valued function. It is unlikely that a direct (noniterative) method can be designed for solving the admittivity imaging problem.
As far as we know, the approach in this paper and the analysis of its convergence and stability have not been reported elsewhere. The convergence result is among the very few algorithms for reconstructing the electrical properties of tissue from internal data.

To formulate  mathematically the imaging problem, we consider a medium of conductivity $\sigma$ and permittivity $\varepsilon$ occupying $\Omega$, $\mathcal{C}^2$-domain of $\mathbb{R}^2$.  (Hereafter, the medium is simply called $\Omega$.)
The problem of micro--electrical impedance tomography is to reconstruct $\sigma$ and $\varepsilon$   from the vector of potential $u_{\omega}$, $\omega \in (\underline \omega, \overline \omega)$, the solution of
\Beq
	\Div (\admitt) \nabla u_{\omega} &=& 0 &\mbox{in } \Omega,\\
	u_{\omega} &=& \varphi &\mbox{on } \partial \Omega,
\Eeq{eq:u}
where $\varphi=(\varphi_1,\varphi_2)\in H^{1/2}(\partial\Omega)^2$.
It is proved in this paper that the above inverse problem is stably solvable with a good choice of boundary datum $\varphi$; that is, $\varphi$ belongs to what we will refer to as the proper set of boundary measurements; see \cite{siap2011, sirev, Triki:ip2010}.

The paper is organized as follows. First, in section \ref{sect2} we review some useful regularity results for elliptic systems of partial differential equations. In section \ref{sec:proper set} we introduce the set of proper boundary measurements. Section \ref{sect4} is devoted to the reconstruction method. We prove that the minimization functional is Fr\'echet differentiable and we compute its derivative. Then we construct an initial guess and prove the convergence of a minimizing sequence. The paper ends with a short discussion.    In the appendix, we prove the convergence of Landweber sequences with cutoff functions.

\section{Preliminaries on regularities} \label{sect2}
Let $\Omega'=\{x\in \Omega~:~ \mbox{dist}(x,\partial\Omega) > c_0\}$ for a small constant $c_0>0$. We assume that $\sigma$ and $\varepsilon$ are constant and known in $\Omega\setminus \Omega'$.  In the following, we let $\sigma_*$ and $\varepsilon_*$, the true conductivity and permittivity of $\Omega$, belong to the convex subset of $H^2(\Omega)^2$ given by
$$ \widetilde{\mathcal{S}} = \{(\sigma, \varepsilon) := (\sigma_0, \varepsilon_0) + (\eta_1, \eta_2) | \,(q_1, q_2) \in \mathcal{S}\},$$
where the positive constants $\sigma_0$ and $\varepsilon_0$ are respectively the conductivity and permittivity
in $\Omega\setminus\Omega'$ and
\begin{equation}\label{eq:mathcalS}
\begin{array}{ll}
\mathcal{S} =& \{(\eta_1,\eta_2)\in  H^2_0(\Omega)^2|~  c_1< \eta_1 + \sigma_0 < c_2,~
  c_1< \eta_2 + \epsilon_0 < c_2,~\textrm{supp}~\eta_j \subset \Omega',\\
&\quad\quad\quad\quad\quad\quad\quad\quad\| \eta_j\|_{H^2(\Omega)} \le c_3\|\eta_j\|_{H^1(\Omega)}, ~\|\eta_j\|_{H^1(\Omega)}\le c_4 ~~\mbox{for }j=1,2~ \}
\end{array}
\end{equation}
with $c_1, c_2, c_4$ and $c_4$ being positive constants and $\textrm{supp}$ denoting the support. In other words, we can write
$
\widetilde{\mathcal{S}} = (\sigma_0, \varepsilon_0) + \mathcal{S}.$
Here, the condition of $\| \eta_j\|_{H^2(\Omega)} \le c_3\|\eta_j\|_{H^1(\Omega)}$ is used to exclude any micro-local oscillation on the admittivity distribution.

Introducing an open subset of $\mathbb{C}$
\begin{equation} \label{defO}
\mathcal{O} := \left\{o \in \mathbb{C} | \Im m \, o  < \frac{c_1}{2 c_2} \right\},
\end{equation}
we first establish a useful lemma, which is a direct consequence of  standard regularity results.
\begin{Lemma}\label{eqn 2.1}
Let $(\sigma,\varepsilon)\in \widetilde{\mathcal{S}}$, $\omega \in \mathcal{O},$ and $f\in L^p(\Omega)$ for $2\le p < \infty$. If $v\in H^1(\Omega)$ satisfy
\begin{equation}
 \Div (\admitt) \nabla v =  f ~~~\mbox{in }~~ \Omega,
\end{equation}
then $v\in W^{2,p}(\Omega')$ and
\beq
	\|v\|_{W^{2,p}(\Omega')} ~~\leq ~~C~ (\|v\|_{L^p(\Omega)} + \|f\|_{L^p(\Omega)}),
\eeq{H2 regularity}
where $C$ depends only on $c_i,i=0, \ldots, 4$, $p$, and  $\Omega.$  Moreover, if $v=0$ on $\partial \Omega$, then
\beq
	\|v\|_{W^{2,p}(\Omega)} ~~\leq ~~C~ (\|v\|_{L^p(\Omega)} + \|f\|_{L^p(\Omega)}).
\eeq{H2 regularity0}
\label{Lemma H2}
\end{Lemma}
\begin{proof}
	From the standard regularity estimate, we have
\beq
	\|v\|_{H^2(\Omega')}~~ \leq ~~C ~(\|f\|_{L^2(\Omega)^2} + \|v\|_{L^2(\Omega)} ).
\eeq{Lax Milgram}
The first equation in \eqref{eqn 2.1} can be rewritten as
\beq
	\Delta v = -\nabla v^T \frac{\nabla (\admitt)}{\admitt} + \frac{f}{\admitt},
\eeq{2.4b} where  $T$ denotes the transpose.
Since $\mbox{supp}~\nabla (\admitt)\subset\Omega'$, we have
 \begin{eqnarray*}
 \|\nabla v^T \frac{\nabla (\admitt)}{\admitt}\|_{L^p(\Omega)} & =& \|\nabla v^T \frac{\nabla (\admitt)}{\admitt}\|_{L^p(\Omega')}\\
&\le & C\|\nabla v^T\|_{L^{2p}(\Omega')^{2}} \| \frac{\nabla (\admitt)}{\admitt}\|_{L^{2p}(\Omega')^2}\\
&\le & C \| v\|_{H^{2}(\Omega')} \| \admitt\|_{H^{2}(\Omega')}\\
&\le & C \left( \| v\|_{L^2(\Omega')}  +  \|f\|_{L^2(\Omega)}\right)  \| \admitt\|_{H^{2}(\Omega')}.
\end{eqnarray*}
Here, Schwartz inequality was used for the second inequality; Sobolev embedding for the third inequality; and the last inequality comes from (\ref{Lax Milgram}).
Hence, the right side of (\ref{2.4b}) is in $L^p(\Omega)$.
Now, we apply the standard $W^{2,p}$-estimate for Poisson's equation (\ref{2.4b}) to get
\begin{eqnarray*}
	\|v\|_{W^{2,p}(\Omega')} &\leq &C
\left( \| v\|_{L^{p}(\Omega)^{2}}  +  \|f\|_{L^p(\Omega)}\right).
\end{eqnarray*}
\end{proof}

\section{Sets of proper boundary conditions} \label{sec:proper set}

The main purpose of this section is to choose ``good" boundary datum $\varphi$ in \eqref{eq:u} so that the measurements of the corresponding vector potential $u_{\omega}$ are helpful in our reconstruction algorithm.  Such a set of good functions, henceforth coined as \textit{a set of proper boundary conditions}, is defined as follows.

\begin{Definition}\label{defi:properbc}
 Let $\varphi \in H^{1/2}(\partial \Omega)^2$. We say that
  $\varphi$ is a proper set of boundary conditions if and only if the $2 \times 2$ matrix $\nabla u_\sigma$ is invertible in
      $\Omega$ for all $\sigma\in \sigma_0+\mathcal S$ where the vector $u_\sigma$ denotes the solution of the boundary value problem
\begin{equation*}
    \left\{
         \begin{array}{ll}
               \vspace{0.2cm} \nabla \cdot \sigma\nabla u = 0 &\hspace{1cm}\mbox{in } \Omega,\\
                 u = \varphi &\hspace{1cm}\mbox{on } \partial \Omega.
            \end{array}
    \right.
\end{equation*}
\end{Definition}
The existence of a set of proper boundary conditions was proved in
\cite{aless, bauman, jinkeun2}.

The following proposition is the main result of this section.
\begin{Proposition}\label{prop:multi}
     For all $(\sigma, \varepsilon) \in \widetilde{\mathcal{S}}$, we denote by $u_{\omega}$ the solution of (\ref{eq:u}) with $\varphi$ being a proper set of boundary conditions. There exist $N>1$ open pairwise disjoint open subsets $B_1, B_2, \cdots, B_N$ of $\Omega$, and $N$ frequencies $\omega_1, \cdots, \omega_N \in (\underline \omega, \overline \omega)$ such that
\begin{itemize}
    \item[(i)] $\overline{\Omega'} \subset \displaystyle\cup_{j = 1}^N \overline B_j  \subset \Omega$;
    \item[(ii)] The matrix $A_{\omega_j}(x) = \nabla u_{\omega}$ is invertible for all $x \in B_j$.
\end{itemize}
\label{Pro Invertibility of A}
\end{Proposition}

In \cite{Alberti}, G. Alberti has proved  the result when the dependence of coefficients on the frequency is different  from that in our context. The key of his arguments is the fact that $u_{\omega}$ is analytic with respect to $\omega.$ Fortunately, his technique is still applicable to \eqref{eq:u}. We present the proof here for the completeness' sake.

\begin{Lemma}\label{prop:analycity}
Let $\mathcal{O}$ be defined by (\ref{defO}). The map
\begin{equation*}
\begin{array}{lclc}
    \vspace{0.2cm} L : &\mathcal{O}&\rightarrow& H^2_{loc}(\Omega)^2,\\
    &\omega &\mapsto& u_{\omega},
    \end{array}
\end{equation*}
where $u_{\omega}$ is the solution to \eqref{eq:u},
 is analytic. Moreover, the derivative of $L$ at $\omega_0$ is given by the solution of
 \Beq
 	\Div (\admittO) \nabla w &=& -\Div i \varepsilon \nabla L(\omega_0) &\mbox{in } \Omega,\\
 	w &=& 0 &\mbox{on } \partial \Omega
 \Eeq{2.3} for all $\omega_0 \in \mathcal{O}.$
\end{Lemma}

\begin{proof}
The quotient \[z := \ds
\frac{L(\omega) - L(\omega_0)}{\omega - \omega_0}\] solves
\Beq
	\Div (\admitt) \nabla z &=& -i\Div \varepsilon \nabla L(\omega_0) & \mbox{in } \Omega,\\
	z &=& 0 &\mbox{on } \partial \Omega.
\Eeq{2.4}
Since $\nabla \cdot \varepsilon \nabla L(\omega_0)=0$ in $\Omega\setminus \overline{\Omega^\prime}$ and $\nabla \cdot \varepsilon \nabla L(\omega_0)$ is in $L^2(\Omega^\prime)$ (see Lemma \ref{Lemma H2}), we can use Lemma \ref{Lemma H2} again to get
\beq
	\|z\|_{H^2(\Omega)} \leq C \|L(\omega_0)\|_{H^2 (\Omega)}
\eeq{2.55}
for some positive constant $C$.

On the other hand, the difference between $z$ and $w$ satisfies
\Beq
	\Div (\admittO) \nabla (z - w) &=& -\Div i (\omega - \omega_0) \varepsilon \nabla z &\mbox{in } \Omega,\\
	z - w &=& 0 &\mbox{on } \partial \Omega,
\Eeq{2.5} where $w$ is defined by \eqref{2.3}.
Applying Lemma \ref{Lemma H2} one more time allows us to obtain
\[
	\|z - w\|_{H^2(\Omega)} \leq C |\omega - \omega_0|\|\nabla z\|_{H^2(\Omega)}.
\]
This, together with \eqref{2.55}, completes the proof of this lemma.
\end{proof}

We are now in position to prove  Proposition \ref{prop:multi}.
\begin{proof}[Proof of Proposition \ref{prop:multi}]  Let $\Omega''=\{x\in \Omega~:~ \mbox{dist}(x,\partial\Omega) > c_0/2\}$, so that $\Omega'\subset\subset\Omega''\subset\subset \Omega$.  From Lemma \ref{eqn 2.1}, $u_{\omega}\in W^{2,p}(\Omega'')$ for any $p>2$.  Hence, it follows from Sobolev embedding that  $u_{\omega}\in \mathcal{C}^{1,\alpha}(\overline {\Omega''})$ for some $\alpha\in (0,1)$.  Thus we can consider $u_{\omega}$ and $\nabla u_{\omega}$ pointwisely. We employ the ideas in \cite{Alberti} to prove the proposition. Since $\det : \mathcal{C}(\overline {\Omega''})^{2 \times 2} \to \mathcal{C}(\overline {\Omega''})$ is multilinear and bounded and
\[
	\begin{array}{llc}
    \vspace{0.2cm}\mathcal{O} &\to& \mathcal{C}^{1,\alpha}(\overline {\Omega''})^2 \\
    \omega & \mapsto & u_{\omega}
    \end{array}
\] is analytic thanks to Lemma \ref{prop:analycity}. Moreover,
\[\begin{array}{llc}
    \vspace{0.2cm}\mathcal{O} &\to& \mathcal{C}^{0,\alpha}(\overline {\Omega''}) \\
    \omega & \mapsto & \det( \nabla u_{\omega})
    \end{array}\] is also analytic.
For $x \in \Omega$, if $\det A_{\omega}(x) = 0$ for every $\omega \in [\underline \omega, \overline \omega]$ then for all $\omega \in \mathcal{O}, \, \det A_{\omega}(x) = 0$ by the analytic continuation theorem. In particular, $\det A_0(x) = 0$ which conflicts with the choice of proper boundary conditions. Hence, we can find $\omega_x \in (\underline \omega, \overline \omega)$ such that $|\det A_{\omega_x}(x)| > 0.$ Moreover, since the map $|\det A_{\omega_x}(\cdot)|$ is continuous, it is strictly positive in the ball $B_{r_x}(x)$, centered at $x$ and of radius $r_x >0$. Noting that $\cup_{x \in \Omega'} B_{r_x}(x)$ covers $\Omega',$ we can use the compactness of $\overline{\Omega'}$ in $\mathbb{R}^2$ to complete the proof.
\end{proof}

From now on, a proper set of boundary conditions $\varphi$ has been chosen. However, in practice, one might not know the values of frequencies and the set $B_1, \cdots, B_N$. We thus suggest to measure the data $u_{\omega}$ for all $\omega \in (\underline \omega, \overline \omega)$. The following corollary of Proposition \ref{prop:multi} will be useful for the sequel.

\begin{Corollary}
    If $\varphi$ is a proper set of boundary conditions then we can find $\lambda > 0$ such that
\begin{equation*}
    \displaystyle\int_{\underline \omega}^{\overline \omega}|\det \nabla u_\omega(x)|dx > \lambda,
\end{equation*}
where $u_\omega(x)$ is the solution of (\ref{eq:u}).
\label{Cor 2.5}
\end{Corollary}

\section{The reconstruction method} \label{sect4}

\subsection{Optimization scheme}

Let the function $U_{\omega} = F[\sigma_*, \varepsilon_*; \omega]$ represent the measurement of the solution vector with
$\sigma_*$ and $\varepsilon_*$ being the true distributions.

Consider $$\begin{array}{lclc}
    \vspace{0.2cm}F : & \widetilde{\mathcal{S}} \times (\underline \omega, \overline \omega) &\to& H^2(\Omega)^2 \\
    &(\sigma, \varepsilon; \omega) & \mapsto & u_{\omega} - U_{\omega},
    \end{array}$$
where again $u_{\omega}$ is the solution to (\ref{eq:u}) with a proper set of boundary conditions $\varphi$. Here $\widetilde{\mathcal{S}}$ is considered as a subset of the Hilbert space $H^2(\Omega)^2$. Note that $F$ is well-defined thanks to Lemma \ref{Lemma H2}.

To reconstruct $\sigma$ and $\varepsilon,$ we minimize the discrepancy functional
\begin{equation*}
    J[\sigma, \varepsilon] = \frac{1}{2} \int_{\underline \omega}^{\overline \omega}
     \|F[\sigma, \varepsilon; \omega] \|^2_{H^1 (\Omega)} d\omega
\end{equation*}
for $(\sigma, \varepsilon) \in \widetilde{\mathcal{S}}$.

We first investigate the differentiability of $F$ with respect to the pair of admittivity $(\sigma, \varepsilon)$. For doing so, we need one more notation. Let  $A:B= \sum_{i,j} a_{ij} b_{ij}$ for two matrices $A=(a_{ij})$ and $B=(b_{ij})$. Let  $\langle \, , \, \rangle_{H^s}$ denote the $H^s(\Omega)^2$-scalar product for $s=1,2$.   The following lemma holds.

\begin{Lemma} \label{lem:DF}
\begin{itemize}
\item[{\rm (i)}] The map $F$ is Fr\'echet differentiable in $(\sigma,\varepsilon) \in \widetilde{\mathcal{S}}$. For all $(h, k) \in \mathcal{S},$ $\D F[\sigma, \varepsilon; \omega](h, k)$ is given by the solution of
\Beq
	\Div (\admitt) \nabla v_{\omega} &=& -\Div (h + i \omega k) \nabla u_{\omega} & \mbox{in } \Omega,\\
	v_{\omega} &=& 0 &\mbox{on } \partial \Omega.
\Eeq{eq:v} Moreover, $DF$ is Lipschitz continuous with respect to $(\sigma, \varepsilon)$.
\item[{\rm (ii)}] $J$ is Fr\'echet differentiable in $(\sigma, \varepsilon) \in \widetilde{\mathcal{S}}.$ Moreover, for all $(h, k) \in \mathcal{S},$
\begin{equation} \label{eq:DJ}
\begin{array}{lll} D J[\sigma, \varepsilon] (h, k)  &=& \ds \Re e  \int_{\underline \omega}^{\overline \omega}  \langle DF[\sigma, \varepsilon; \omega](h,k), F[\sigma, \varepsilon;\omega] \rangle_{H^1},\\
\nm &=& \ds \Re e  \int_{\underline \omega}^{\overline \omega}  \langle (h,k), D F[\sigma, \varepsilon; \omega]^* (F[\sigma, \varepsilon;\omega]) \rangle_{H^2},
\end{array}
\end{equation}
where $D F[\sigma, \varepsilon;\omega]^*$ is the adjoint of $DF[\sigma, \varepsilon;\omega]$.
\item[{\rm (iii)}]  Furthermore, for all $(h, k) \in \mathcal{S},$
\begin{equation} \label{eq:DJ2}  D J[\sigma, \varepsilon] (h, k)
= \Re e  \int_{\underline \omega}^{\overline \omega} \int_\Omega (h + i \omega k) \nabla u_\omega :\nabla p_\omega \, d\omega, \end{equation}
where $p_\omega \in H^2(\Omega)$ is the solution to the adjoint problem
\begin{equation}\label{eq:dualp}
    \left\{
        \begin{array}{ll}
            \nabla \cdot(\sigma + i\omega \varepsilon)\nabla p_{\omega} =
             \overline{F(\sigma, \varepsilon; \omega)} - \Delta \overline{F(\sigma, \varepsilon; \omega)} &\mbox{in } \Omega,\\
            p_{\omega} = 0 &\mbox{on } \partial \Omega.
        \end{array}
    \right.
\end{equation}

\end{itemize}
\end{Lemma}
\begin{proof}
Take $(h, k) \in \mathcal {S}$ such that $(\sigma + h, \varepsilon + k)$ still belongs to $\widetilde{\mathcal{S}}$.
Define \[w_{h, k} = F[\sigma + h, \varepsilon + k;\omega] - F[\sigma, \varepsilon;\omega] \in H_0^1(\Omega)^2.\] We have
\begin{eqnarray*}
	\Div (\sigma + h + i \omega (\varepsilon + k)) \nabla w_{h, k} &=& -\Div (\sigma + h + i \omega (\varepsilon + k)) \nabla (F[\sigma, \varepsilon;\omega] + U_\omega)\\
	&=& \Div (h + i \omega k) \nabla (F[\omega, \sigma, \varepsilon] + U_\omega).
\end{eqnarray*}
Using Sobolev embedding and Lemma \ref{Lemma H2}, we have
\beq
\begin{array}{ll}
 	\|w_{h, k}\|_{H^2(\Omega)^2} &\leq C  \|\nabla\cdot( h + i \omega  k) \nabla (F[\sigma, \varepsilon;\omega] + U_\omega) \|_{L^2(\Omega')^2}\\
 &\leq C \Big( \|h + i \omega  k\|_{L^\infty(\Omega')} \| F[\sigma, \varepsilon;\omega] + U_\omega\|_{H^2(\Omega')^2}  \\
 &~~~~\quad\quad + \|\nabla (h + i \omega  k)\|_{L^4(\Omega')^2}\| \nabla (F[\sigma, \varepsilon;\omega]+ U_\omega) \|_{L^4(\Omega')^{2\times 2}}    \Big)\\
 &\leq C \left(\|h\|_{H^2(\Omega)} + \|k\|_{H^2(\Omega)}\right) \left( \|F[\omega, \sigma, \varepsilon]\|_{H^2(\Omega')^2} + \| U_\omega \|_{H^2(\Omega')^2} \right).
\end{array}
\eeq{Sobolev-estimate}

The function $w_{h, k} - v_{\omega} \in H^1_0(\Omega)$ and satisfies
\begin{eqnarray*}
	\Div (\admitt) \nabla (w_{h, k} - v_{\omega}) &=&  - \nabla \cdot (h + i\omega k)
	\nabla w_{h, k}.
\end{eqnarray*}
Thus, again by repeating the estimates as in (\ref{Sobolev-estimate}), we get
\begin{eqnarray*}
	\|w_{h, k} - v_{\omega}\|_{H^2(\Omega)^2}
	&\leq& C \left( \|h\|_{H^2(\Omega)} + \|k\|_{H^2(\Omega)}\right)\|w_{h, k}\|_{H^2( \Omega')^2}\\
&\leq& C \left( \|h\|_{H^2(\Omega)} + \|k\|_{H^2(\Omega)}\right)^2 \left(\|F[\omega, \sigma, \varepsilon]\|_{H^2( \Omega')^2} +  \| U_\omega \|_{H^2(\Omega')^2} \right).
\end{eqnarray*}
Item (i) has  been then proved.  Moreover, it is easy to see that $DF$ is Lipschitz continuous with respect to $(\sigma, \varepsilon)$. In fact,
let $(\sigma, \varepsilon)$ and $(\sigma', \varepsilon')$ be in $\widetilde{\mathcal{S}}$. Let $(h, k)$ be in $\mathcal{S}$. Then, $DF[\sigma, \varepsilon;\omega](h, k) - DF[\sigma', \varepsilon'; \omega](h, k)$ is solution to the following equation:
\begin{equation*}
\left \{
\begin{array}{l}
\vspace{0.2cm} \nabla \cdot (\sigma + i \omega \varepsilon) \nabla \left(DF[\sigma, \varepsilon; \omega](h, k) - DF[\sigma', \varepsilon'; \omega](h, k) \right) = \\
\vspace{0.3cm}\hspace{1cm} - \nabla \cdot (h + i\omega k) \nabla (F[\sigma, \varepsilon;\omega] - F[\sigma', \varepsilon';\omega]) \\
\vspace{0.3cm}\hspace{1cm} - \nabla \cdot (\sigma - \sigma' + i \omega (\varepsilon - \varepsilon'))\nabla DF[\sigma', \varepsilon';\omega](h,k)\quad \textrm{in}\, \Omega, \\
DF[\sigma, \varepsilon;\omega](h,k) - DF[\sigma', \varepsilon';\omega](h,k)= 0 \quad \textrm{on}\, \partial \Omega.
\end{array}
\right .
\end{equation*}
Therefore, applying similar estimate as in (\ref{Sobolev-estimate}), we have
\begin{equation}\label{Es1}\begin{array}{l}
\vspace{0.2cm} \| (DF[\sigma, \varepsilon;\omega] - DF[\sigma', \varepsilon';\omega]) (h,k)\|_{H^2(\Omega)^2} \\
\vspace{0.2cm}\hspace{2.5cm}\leq C \left (\|h + i \omega k\|_{H^2(\Omega)} \| F[\sigma , \varepsilon;\omega] - F[\sigma', \varepsilon';\omega])\|_{H^2(\Omega')^2} \right.\\
\vspace{0.2cm} \hspace{3.5cm}\left.+ \|\sigma - \sigma' + i \omega (\varepsilon - \varepsilon')\|_{H^2(\Omega)}\| DF[\sigma', \varepsilon';\omega](h,k)\|_{H^2(\Omega)^2}\right).
\end{array}
\end{equation}
Since $F[\sigma , \varepsilon;\omega] - F[\sigma', \varepsilon';\omega]$ satisfies
\begin{eqnarray*}
	\Div (\sigma+ i \omega \varepsilon) \nabla (F[\sigma , \varepsilon;\omega] - F[\sigma', \varepsilon';\omega])= -\Div (\sigma - \sigma' + i \omega (\varepsilon - \varepsilon')) \nabla ( F[\sigma', \varepsilon';\omega] + U_\omega),
\end{eqnarray*}
we apply a similar estimate as in (\ref{Sobolev-estimate}) to get  Lipschitz continuity of $F$:
\begin{equation}\label{Es2}
\begin{array}{lll}
\| F[\sigma , \varepsilon;\omega] - F[\sigma', \varepsilon';\omega])\|_{H^2(\Omega')}&\le& C \|\sigma - \sigma' + i \omega (\varepsilon - \varepsilon')\|_{H^2(\Omega)} \\
\nm
&& \qquad \times \left( \|F[\sigma', \varepsilon';\omega]\|_{H^2( \Omega')^2} +  \| U_\omega \|_{H^2(\Omega')^2} \right).
\end{array}
\end{equation}
Noting that $DF[\sigma, \varepsilon;\omega](h,k)$ is the solution of (\ref{eq:v}), we also have
\begin{equation}\label{Es3}
\| DF[\sigma', \varepsilon';\omega](h,k)\|_{H^2(\Omega')}\le C \|h + i \omega k\|_{H^2(\Omega)}\|\left( F[\sigma', \varepsilon';\omega]\|_{H^2( \Omega')^2} + \| U_\omega \|_{H^2(\Omega')^2} \right).
\end{equation}
Hence, combining estimates (\ref{Es1})-(\ref{Es3}), we have
$$\begin{array}{lll}
 \| DF[\sigma, \varepsilon;\omega] - DF[\sigma', \varepsilon';\omega]\|_{\mathcal L(H^2(\Omega), H^2(\Omega)) }
 &\le & C  \|\sigma - \sigma' + i \omega (\varepsilon - \varepsilon')\|_{H^2(\Omega)}\\
 \nm
 && \times
\left( \| F[\sigma', \varepsilon';\omega]\|_{H^2(\Omega')^2} + \| U_\omega \|_{H^2(\Omega')^2} \right).
\end{array}
$$

    Item (ii) can be easily proved by using arguments similar to those used above. Item (iii) follows by integration by parts.
\end{proof}

We can now apply the gradient descent method to minimize the discrepancy functional $J$. We compute the iterates
\begin{equation} \label{defeta}
(\sigma_{n+1}, \varepsilon_{n+1}) = T[\sigma_n, \varepsilon_n] - \mu D
J[T[\sigma_n,\varepsilon_n]],
\end{equation}
where $\mu >0$ is the step size and
$T[f]$ is any approximation of the Hilbert projection from $H^2(\Omega)^2$ onto $\overline{\widetilde{\mathcal{S}}}$ with $\overline{\widetilde{\mathcal{S}}}$ being the closure of $\widetilde{\mathcal{S}}$ (in the $H^2$-norm). The derivative $D
J[T[\sigma_n, \varepsilon_n]]$
is given by
$$
 DJ[T[\sigma_n,\varepsilon_n]]=
 (- \Re e \, \nabla u_\omega : \nabla p_\omega, \omega \Im m \, \nabla u_\omega :
 \nabla p_\omega),$$
 where $u_\omega$ and $p_\omega$ are respectively the solutions to (\ref{eq:u}) and (\ref{eq:dualp}) with $(\sigma, \varepsilon)=
 T[\sigma_n,\varepsilon_n]$.

The presence of $T$ is necessary because $
(\sigma_n, \varepsilon_n)$ might not be in $\widetilde{\mathcal{S}}$.

Using (iv) we can show that the optimal control algorithm (\ref{defeta}) is nothing else than the following Landweber scheme \cite{landweber, Hankeetal:nm1995} given by
\beq\begin{array}{rcl}
\ds(\sigma_{n+1}, \varepsilon_{n+1}) \\&& \ds\hspace*{-.8in} = T[\sigma_n, \varepsilon_n] - \mu \int_{\underline \omega}^{\overline \omega}\D
F^*[T[\sigma_n, \varepsilon_n];\omega] (F[T[\sigma_n, \varepsilon_n];\omega]) \, d\omega.
\end{array}
\eeq{minimizing sequence}

  \subsection{Initial guess}

  To initialize the previous optimal control algorithm, we need to construct an initial guess for the electrical property distributions $\sigma$ and $\epsilon$.

 Consider the solution $u_\omega$ to (\ref{eq:u}).
 For all $x \in \Omega$,
\[
    \Delta u_{\omega} + \frac{\nabla u_{\omega}^T\nabla (\sigma + i\omega \varepsilon)}{\sigma + i\omega \varepsilon} = 0.
\]
It follows that
\begin{equation}\label{eq:initeq}
    A_{\omega}^T \frac{\nabla (\sigma + i\omega \varepsilon)}{\sigma + i\omega \varepsilon} = -\nabla \cdot A_{\omega},
\end{equation}
where
\[
    A_{\omega} = \nabla u_{\omega}.
\]
Equation (\ref{eq:initeq}) gives us several ways to  reconstruct both
$\sigma$ and $\varepsilon$. We suggest to define the map
$\gamma_{\omega} = \log (\sigma + i\omega \varepsilon)$, whose
imaginary part is chosen in $[0, \frac{\pi}{2})$, and solve
\begin{equation} \label{eq:gamma2}
 \left\{ \begin{array}{ll}
    \vspace{0.2cm}\displaystyle\Delta \gamma_{\omega} = \nabla \cdot (- (\overline{A_{\omega}} A_{\omega}^T)^{\dagger} \overline{A_{\omega}}
    \nabla \cdot A_{\omega}) &\hspace{1cm}\mbox{in } \Omega,\\
    \displaystyle\gamma_{\omega} = \log(\sigma_0 + i\omega \varepsilon_0) &\hspace{1cm}\mbox{on } \partial
    \Omega,
    \end{array}
    \right.
\end{equation}
where  $\dagger$ denotes the
pseudo-inverse. The knowledge of $\gamma_{\omega}$ implies those of $\sigma$ and
$\varepsilon$. We denote by $\sigma_I$ and $\varepsilon_I$ the obtained
functions by averaging $\gamma_{\omega}$ over $\omega$:
\begin{equation*}
\displaystyle\sigma_I + i \frac{(\overline{\omega}) + \underline{\omega}}{2}
\varepsilon_I = \frac{1}{\overline{\omega} - \underline{\omega}}
\int_{\underline \omega}^{\overline \omega}  e^{\gamma_\omega}
d\omega,
\end{equation*}
where $\gamma_\omega$ is given by (\ref{eq:gamma2}).  We use
$\sigma_I$ and $\varepsilon_I$ as the initial guess for our desired
coefficients.

\subsection{Convergence of the minimizing sequence}

The following theorem holds.
 \begin{Theorem}
 For all $(h, k) \in \mathcal{S},$ we have the following estimate:
\begin{equation}
    \int_{\underline \omega}^{\overline \omega} \| \D F[\sigma,
    \varepsilon; \omega] (h, k)\|_{H^1(\Omega)^2} d\omega \geq C\|(h, k)\|_{H^2(\Omega)^2}
\label{3.8}
\end{equation}
for some positive constant $C$.
\label{Theorem 4.1}
\end{Theorem}

 \begin{proof}

  Assume to the contrary that \eqref{3.8} is not true. That means we can find $h_n$ and $k_n$ in $\mathcal{S}$ such that
 \[
 	\|h_n\|_{H^2(\Omega)} + \|k_n\|_{H^2(\Omega)} = 1
 \] and
 \[
 	\intome \|D F[\sigma, \varepsilon;\omega](h_n, k_n)\|_{H^1(\Omega)} d\omega \rightarrow 0
 \] as $n \rightarrow \infty$.
By compactness, up to extracting a subsequence, we can assume that
\beq
(h_n, k_n) \rightharpoonup (h, k) \quad \mbox{in } H_0^1(\Omega)^2 .
\eeq{5.3}

Denote by $u_{\omega}$ the vector  $F[\sigma, \varepsilon;\omega]$ and $v_{\omega}^n$ the vector $D F[\sigma, \varepsilon;\omega](h_n, k_n)$. We have \[v_{\omega}^n \rightarrow 0 \quad \mbox{in } H^1_0(\Omega)\] for all $\omega \in (\underline \omega, \overline \omega).$

Recall $N, B_1, \cdots, B_N, \omega_1, \cdots, \omega_N$, as in Proposition \ref{prop:multi}. Fixing $j \in \{1, \cdots, N\}$, we have
\begin{eqnarray*}
	-\Div (\admittj) \nabla v_{\omega_j}^n &=& \Div (h_n + i \omega_j k_n) \nabla u_{\omega_j} \\
	&=& (\admittj)  \nabla u_{\omega_j}^T \nabla \frac{h_n + i\omega_j k_n}{\admittj}
\end{eqnarray*} in $B_j$.
Equivalently,
$$
\nabla u_{\omega_j}^T \nabla \frac{h_n + i\omega_j k_n}{\admittj} = - \nabla \log (\admittj) \cdot \nabla v_{\omega_j}^n - \Delta v_{\omega_j}^n.
$$
Note that the left-hand side of the equation above tends to $0$ in $H^{-1}(\Omega)$,
so is $\nabla \frac{h_n + i\omega_j k_n}{\admittj}$ in $L^2(B_j)$.
By using Poincar\'e's inequality and the fact that $\overline {\Omega'} \subset \cup_{j=1}^N \overline B_j$,
we arrive at $h=k=0$.
Since $(h_n,k_n)\in \mathcal S$, $\|h_n\|_{H^2(\Omega)} + \|k_n\|_{H^2(\Omega)} \rightarrow 0$, which contradicts the assumption.
 \end{proof}

Note that as a direct consequence of Theorem \ref{Theorem 4.1}, it follows that
\beq%
	 \left(\int_{\underline \omega}^{\overline \omega} \|\D F[\sigma,
    \varepsilon;\omega] (h, k)\|_{H^1(\Omega)^2}^2 d\omega\right)^{\frac{1}{2}} \geq C\|(h, k)\|_{H^2(\Omega)^2}
\eeq{}
for some positive constant $C$.
Hence, Theorem \ref{Theorem 4.1} and Proposition \ref{proposition Landweber} yield our main result in this paper.
\begin{Theorem}
	The sequence defined in \eqref{minimizing sequence} converges to the true admittivity $(\sigma_*, \varepsilon_*)$ of $\Omega$ in the following sense: there is $\eta>0$ such that if
	$\| T[\sigma_I, \varepsilon_I] -  (\sigma_*, \varepsilon_*)\|_{H^2(\Omega)^2} < \eta$, then
	$$ \lim_{n \rightarrow + \infty} \| \varepsilon_n - \varepsilon_*\|_{H^2(\Omega)} + \| \sigma_n - \sigma_*\|_{H^2(\Omega)} =0.$$
\end{Theorem}

%

\section{Concluding remarks}

 In this paper we have proposed for the first time an optimal control algorithm for admittivity imaging from multi-frequency micro-electrical data. We have proved its convergence and its local stability. Our approach in this paper can be extended to elastography and can be used to image both shear modulus and viscosity tissue properties from internal displacement measurements. Another interesting problem is to image  tissues with anisotropic impedance distribution from micro-electrical data.

\appendix
\section{The convergence of the Landweber sequence with a Hilbert projection}
This appendix follows from \cite{Hankeetal:nm1995}; see also \cite{otmar}. It proves the convergence of
the Landweber scheme with a Hilbert projection.

Let $X$ and $Y$ be Hilbert spaces and $F:  K \times (\underline \omega, \overline \omega) \rightarrow Y$ be a differentiable map where $K$ is a convex subset of $X$. Let $\langle \, ,\, \rangle_X$ and
$\langle \, ,\, \rangle_Y$ denote the scalar products in $X$ and $Y$, respectively.

We are interested in solving the equation
\beq
	F[ x_*;\omega] = 0 \quad \mbox{for all } \omega \in (\underline \omega, \overline \omega).
\eeq{}
It is natural to minimize
\beq
	J[x] = \frac{1}{2}\intome\|F[x;\omega] \|_Y^2 d \omega,
\eeq{} with $x \in K$.
Assume that $F[\cdot;\omega]$  is Fr\'echet differentiable. So is $J$. The derivative of $J$ is given by
\begin{eqnarray*}
	\Dx J[x](h) &=& \intome \langle \Dx F[x;\omega](h), F[x;\omega] \rangle_Y d\omega \nonumber \\ &=& \intome\langle h, \Dx F[x;\omega]^* (F[x;\omega])\rangle_X d\omega,
\end{eqnarray*} where the superscript $^*$ indicates the dual map. The iteration sequence due to the descent gradient method is given by
\beq
	x_{n + 1} = T[x_n] - \mu \intome \Dx F[T[x_n];\omega]^* (F[T[x_n];\omega]) \, d\omega.
\eeq{A4}
Here, $\mu$ is a small number and $T[x] \in K$ is an approximation of the Hilbert projection of $X$ onto $\overline K$
\beq
	 P: X \ni x \mapsto {\rm argmin}\{\|x - a\|: a \in \overline K\}.
\eeq{}
Without loss of generality, we can assume that
\[
	\|T[x_n] - P[x_n]\|_X \leq 2^{-n}, \quad n \geq 1.
\]

The presence of $T$ in \eqref{A4} is necessary because $x_n$ might not be in $K$ and $F[x_n]$ might not be well-defined. The map $T$ above also increases the rate of convergence of $(x_n)$ to $x_*$ due to
\beq
	\|T[x_n] - x_*\|_X \leq \|x_n - x_*\|_X + 2^{-n} , \quad n \geq 1.
\eeq{}

The following proposition holds.
\begin{Proposition}
	Assume that $\Dx F[x;\omega]$ is Lipschitz continuous and that, for all $x, h \in K,$
		\beq
			\intome \|\Dx F[x;\omega](h)\|_{Y}^2d \omega \geq  c\|h\|_X^2.
		\eeq{A6} Then the sequence defined in \eqref{A4} converges to $x_*$ provided that $x_0$ is a "good" initial guess for $x_*$ and $\mu$ is sufficiently small.
\label{proposition Landweber}
\end{Proposition}
\begin{proof} Since $\Dx F[x;\omega]$ is Lipschitz continuous, for
	  all $x$ such that $\|x - x_*\|_X <\eta$ with $\eta$ being a small positive number, we have
		\begin{eqnarray}
			&&\intome \|F[x;\omega] - F[x_*;\omega] - \D F[x;\omega](x - x_*)\|_Y^2 d\omega \nonumber\\&& \hspace{1.8in} \leq C \eta^2 \|x - x_*\|_X^2 \nonumber\\
			&&\hspace{1.8in} \leq C \eta^2 \intome \|F[x;\omega] - F[x_*;\omega]\|_Y^2d \omega \label{local Landweber condition}
		\end{eqnarray}
		for some positive constant $C$.
Note that we have used here \eqref{A6} and the mean-value theorem for the second inequality above.	
		
For all $n \geq 1$, let
\[
\epsilon_n[\omega] =  F[T[x_n];\omega].
\] We have
\begin{eqnarray*}
	&&\hspace*{-.24in}\|x_{n + 1} - x_*\|_X^2 - \|x_n - x_*\|_X^2 - 2^{-n}\\ && \leq
	\|x_{n + 1} - x_*\|_X^2 - \|T[x_n] - x_*\|_X^2 \\
	&&=  2\langle x_{n + 1} - T[x_n], T[x_n] - x_*\rangle_X + \|x_{n + 1} - T[x_n]\|_X^2\\
	&&\leq 2\mu \intome \langle -\Dx F[T[x_n];\omega]^*\epsilon_n[\omega], T[x_n]
	- x_*\rangle_X d\omega \\
		&& \hspace*{.24in}+ \intome \langle \mu \epsilon_n[\omega], \mu \D F[T[x_n];\omega] \D F[ T[x_n]; \omega]^* (\epsilon_n[\omega]) \rangle_Y  d\omega \\
	&& = \intome \langle \epsilon_{n}[\omega], 2\mu \epsilon_n[\omega] -2\mu\D F[T[x_n];\omega] (T[x_n] - x_*)\rangle_Y  d\omega   -  \mu \intome \|\epsilon_n[\omega] \|_Y^2  d\omega \\
		&& \hspace*{.24in}+ \intome \langle \sqrt{\mu} \epsilon_n[\omega], (- I + \mu \D F[T[x_n];\omega]\D F[T[x_n];\omega]^*)) (\sqrt{\mu}\epsilon_n[\omega])\rangle_Y  d\omega \\
		&& \leq 2\mu \left(\intome \|\epsilon_{n}[\omega]\|_Y^2  d\omega\right )^{\frac{1}{2}} \left(\intome \| \epsilon_n[\omega] -\Dx F[T[x_n];\omega] (T[x_n] - x_*)\|_Y^2  d\omega\right)^{\frac{1}{2}} \\
		&& \hspace*{.24in}  -  \mu \intome \|\epsilon_n[\omega]\|_Y^2  d\omega + \intome \langle \sqrt{\mu} \epsilon_n[\omega], (- I + \mu \D F[T[x_n];\omega]\Dx F[T[x_n];\omega]^*)) (\sqrt{\mu}\epsilon_n[
		\omega])\rangle_Y  d\omega \\
		&\leq& \mu (2 \sqrt{C} \eta - 1) \intome\|\epsilon_n[\omega]\|_Y^2  d\omega.
\end{eqnarray*}
Here, we have used (\ref{local Landweber condition}) for the last inequality.
It follows that
\beq
	\|x_{n + 1} - x_*\|_X^2 + \mu(1 - 2 \sqrt{C}\eta) \intome\|\epsilon_n\|_Y^2  d\omega  - 2^{-n} \leq \|x_n - x_*\|_X^2 , \nonumber
\eeq{} and therefore,
\beq
	\sum_{n = 1}^{\infty} \intome \|F[T[x_n];\omega] \|_Y^2 d\omega  \leq \frac{\|x_0 - x_*\|_X^2}{\mu(1 - 2\sqrt{C} \eta)} + 1 . \nonumber
\eeq{}
We now obtain the convergence of $(x_n)$ to $x_*$ using again the mean-value theorem and condition \eqref{A6}:
\beq
	c \|T[x_n] - x_*\|_X^2 \leq \intome \| \Dx F[\tilde x_n; \omega] (T[x_n] - x_*)\|^2_Y d\omega = \intome \|F[T[x_n];\omega] - F[x_*;\omega] \|_Y^2  d\omega \rightarrow 0 \nonumber
\eeq{} for some $\tilde x_n = t T[x_n] + (1 - t)x_*$, $t \in (0, 1)$.
\end{proof}

\end{document}